\documentclass{article}
\author{Peyman Ghahremani\footnote{NSF for support through grant DMS-1500264 is gratefully acknowledged}}
\title{A Short Proof of the Bernstein Inequality for Formal Power Series}
\date{\vspace{-5ex}}
\usepackage{graphicx}

\usepackage{amsthm}
\usepackage[utf8]{inputenc}
\usepackage[english]{babel}
\usepackage{amsthm,amsmath}
\newcounter{dummy} \numberwithin{dummy}{section}
\newtheorem{thm}[dummy]{Theorem}
\newtheorem{defn}[dummy]{Definition}
\newtheorem{lemma}[dummy]{Lemma}

\newtheorem{prop}[dummy]{Proposition}

\usepackage{geometry}                		
\usepackage{graphicx}				
\usepackage{amssymb}
\usepackage{amsmath}
\usepackage{gensymb}
\usepackage{tikz}
\usepackage{enumerate}
\usetikzlibrary{matrix,arrows,decorations.pathmorphing}
\usepackage{xcolor}

\begin{document}
\maketitle

\begin{abstract}
Let $k$ be a field of characteristic zero, let $R$ be the ring of formal power series in $n$ variables over $k$ and let $D(R,k)$ be the ring of $k-$linear differential operators in $R$. If $M$ is a finitely generated $D(R,k)-$module then $d(M)\geq n$ where $d(M)$ is the dimension of $M$. This inequality is called the Bernstein inequality. We provide a short proof.
\end{abstract}

\section{Introduction}
Throughout this paper $k$ is a field of characteristic zero. Let $R=k[[x_1,\cdots,x_n]]$ be the ring of formal power series in $n$ variables over  $k$ and let $D=D(R,k)$ be the ring of  $k-$linear differential operators of $R$. A celebrated result, known as the Bernstein inequality, says the following.

\begin{thm}\label{main}
If $M$ is a finitely generated left $D-$module, then $$d(M)\geq n,$$ where $d(M)$ is the dimension of $M$.
\end{thm}

For the ring of polynomials there is now a very short and simple proof of this result \cite[9.4.2]{Cout} due to A. Joseph. But for formal power series the only known (to the author) proof, due to Bjork \cite[2.7.2]{Bjork}, is much more complicated. It uses some sophisticated homological algebra including spectral sequences. 

Since the formal power series case is important in applications of $D$-modules to commutative algebra, it would be useful to have a simple proof of the Bernstein inequality in this case. This is accomplished in this paper. To emphasize the elementary nature of the proof this paper has been made self-contained modulo some basic commutative algebra. 

Our proof is inspired by \cite{einstein} and is from the author's thesis. I would like to thank my adviser, Professor Gennady Lyubeznik for guidance and support.

\section{Preliminaries}
Keeping the notation of the Introduction, $R$ is a subring of Hom$_k(R,R)$, the $k$-linear endomorphisms of $R$, via the map that sends $r\in R$ to the map $R\to R$ which is the multiplication by $r$ on $R$. Let  $d_i=\frac{\partial}{\partial x_i}:R\to R$ for $1\leq i\leq n$ be the $k$-linear partial differentiation with respect to $x_i$. The ring $D$ of $k$-linear differential operators of $R$ is the subring of the ring of $k$-linear endomorphisms Hom$_k(R,R)$ of $R$ generated by $R$ and $d_1,\dots, d_n$. 

Standard theorems of calculus imply that $d_i$ commutes with $d_j$ and $x_j$ for $j\ne i$. The product formula says that $d_i(fg)=d_i(f)g+fd_i(g)$ for every $f,g\in R$. Fixing $f$ and viewing this as equality between two differential operators applied to $g$, we get the following equality of differential operators
\begin{equation}\label{E: first}
d_if=fd_i+\frac{\partial f}{\partial x_i}.
\end{equation}

An element of $D$ is a finite sum with coefficients in $k$ of products of $x_1,\dots,x_n,d_1,\dots,d_n$. Using commutativity of $d_i$ with $x_j$ and $d_j$ for $j\ne i$ and (\ref{E: first}) with $f=x_i$ one can rewrite every product of $x_1,\dots,x_n,d_1,\dots,d_n$ as a sum of monomials $x_1^{\alpha_1}\cdots x_n^{\alpha_n}d_1^{\beta_1}\cdots d_n^{\beta_n}$ with coefficients in $k$. Collecting similar terms one can write every element of $D$ as a sum $\Sigma_{i_1,\dots,i_n}r_{i_1,\dots,i_n}d_1^{i_1}\cdots d_n^{i_n}$ where $r_{i_1,\dots,i_n}\in R$. Let $\{i_1,\dots, i_n\}$ be the index with smallest total degree $i_1+\dots+i_n$ in such a sum. Applying this sum to $x_1^{i_1}\cdots x_n^{i_n}$ one gets $$\Sigma_{i_1,\dots,i_n}r_{i_1,\dots,i_n}d_1^{i_1}\cdots d_n^{i_n}(x_1^{i_1}\cdots x_n^{i_n})=r_{i_1,\dots,i_n}i_1!\cdots i_n!.$$ Hence $\Sigma_{i_1,\dots,i_n}r_{i_1,\dots,i_n}d_1^{i_1}\cdots d_n^{i_n}= 0$ if and only if every $r_{i_1,\dots,i_n}=0$. Thus the ring $D$ is a free left $R$-module on the monomials $d_1^{\alpha_1}\cdots d_n^{\alpha_n}$ with $\alpha_j\geq 0$. 

There is a filtration $\Sigma=\Sigma_0\subseteq\Sigma_1\subseteq\Sigma_2\subseteq\cdots$ on $D$ where each $\Sigma_i$ is the free $R-$submodule of $D(R,k)$ on the monomials $d_1^{\alpha_1}\cdots d_n^{\alpha_n}$ with $\sum\alpha_j\leq i$. Clearly $$\texttt{gr}_\Sigma D=\Sigma_0\oplus(\Sigma_1/\Sigma_0)\oplus(\Sigma_2/\Sigma_1)\oplus\cdots=R[\bar{d}_1,\cdots,\bar{d}_n]$$ is the ring of polynomials over $R$ in $\bar d_1,\dots,\bar d_n$ where $\bar{d}_i$ is the image of $d_i$ in $\Sigma_1/\Sigma_0$ for $1\leq i\leq n$. This ring $\texttt{gr}_\Sigma D$ is graded in the usual way, i.e. deg$r=0$ for $r\in R$ and deg$\bar d_i=1$ for all $i$.

Now let $M$ be a left $D(R,k)-$module. A good filtration $\Gamma$ for $M$ is an ascending chain $\Gamma_0\subseteq\Gamma_1\subseteq\Gamma_2\subseteq\cdots$ of $R-$submodules of $M$,  such that $\bigcup_{i=0}^\infty\Gamma_i=M$, $\Sigma_i\Gamma_j\subseteq\Gamma_{i+j}$ and $\texttt{gr}_\Gamma(M)=\Gamma_0\oplus(\Gamma_1/\Gamma_0)\oplus(\Gamma_2/\Gamma_1)\oplus\cdots$ is a finitely generated $\texttt{gr}_\Sigma D-$module. A good filtration exists provided $M$ is finitely generated. For example, if $M$ is generated by elements $m_1,\dots, m_s\in M$, then the filtration with $\Gamma_i=\Sigma_im_1+\dots+\Sigma_im_s$ is a good filtration. 

The following proposition is well-known \cite[1.3.4]{Bjork}   $\mathrm{and}$    \cite[11.1.1]{Cout}.

\begin{prop}
Let $M$ be a finitely generated left $D(R,k)-$module. Let $\Gamma=\Gamma_0\subseteq\Gamma_1\subseteq\Gamma_2\subseteq\cdots$ and $\Omega=\Omega_0\subseteq\Omega_1\subseteq\Omega_2\subseteq\cdots$ be two good filtrations on $M$.  Then:
 \item(i)\label{partone} There exists a number $w$ such that $\Omega_{j-w}\subseteq\Gamma_j\subseteq\Omega_{j+w}$ for all $j\geq 0$.
 \item(ii)\label{parttwo} $\sqrt{\texttt{Ann}_{R[\bar{d}_1,\cdots,\bar{d}_n]}(\texttt{gr}_{\Gamma}(M))}=\sqrt{\texttt{Ann}_{R[\bar{d}_1,\cdots,\bar{d}_n]}(\texttt{gr}_{\Omega}(M))}$
\end{prop}
\begin{proof} $(i)$ Let $\sigma_i:\Sigma_i\rightarrow \Sigma_i/\Sigma_{i-1}$ and $\mu_i:\Gamma_i\rightarrow \Gamma_i/\Gamma_{i-1}$ be the natural surjections.
  Since $\texttt{gr}_{\Gamma}(M)$ is a finitely generated $R[\bar{d}_1,\cdots,\bar{d}_n]-$module, there exist integers $k_1,\cdots,k_s$ and elements $u_i\in\Gamma_{k_i}\setminus\Gamma_{k_i-1}$ with $1\leq i\leq s$ such that $\mu_{k_1}(u_1),\cdots,\mu_{k_s}(u_s)$ generate $\texttt{gr}_\Gamma(M)$ as an $R[\bar{d}_1,\cdots,\bar{d}_n]-$module. Let $k_0=\texttt{max}\{k_1,\cdots,k_s\}$. For $v\geq k_0$ we set $R_v=\Sigma_v\Gamma_0+\cdots+\Sigma_{v-k_0}\Gamma_{k_0}$. We show that $\Gamma_v\subseteq R_v$ for $v\geq k_0$ by induction. Clearly $\Gamma_{k_0}=R_{k_0}$. Assume that $v>k_0$ and $\Gamma_{v-1}\subseteq R_{v-1}$. Let $\gamma\in\Gamma_v$. There exist $b_{v-k_1}\in\Sigma_{v-k_1},\cdots,b_{v-k_s}\in\Sigma_{v-k_s}$ such that 
$$\mu_v(\gamma)=\sigma_{v-k_1}(b_{v-k_1})\mu_{k_1}(u_1)+\cdots+\sigma_{v-k_s}(b_{v-k_s})\mu_{k_s}(u_s)$$

\noindent
thus 
$$\gamma\in\Sigma_{v-k_1}\Gamma_{k_1}+\cdots+\Sigma_{v-k_s}\Gamma_{k_s}+\Gamma_{v-1}\subseteq R_v+\Gamma_{v-1}\subseteq R_v+R_{v-1}=R_v$$
whence $\Gamma_v\subseteq R_v$. This completes the induction.

\noindent
Clearly $\{\Omega_j\cap\Gamma_{k_0}\}_{j\geq 0}$ forms an increasing chain of the finitely generated $R-$module $\Gamma_{k_0}$. Since $\bigcup_{j\geq 0}\Omega_j=M$ and $R$ is Noetherian we have $\Gamma_{k_0}\subseteq\Omega_{w'}$ for some number $w'$. If $0\leq j\leq k_0$ and $v\geq k_0$ we have 
$$\Sigma_{v-j}\Gamma_j\subseteq\Sigma_{v-j}\Gamma_{k_0}\subseteq\Sigma_{v-j}\Omega_{w'}\subseteq\Sigma_v\Omega_{w'}\subseteq\Omega_{v+w'}$$

\noindent
This shows that $\Gamma_v\subseteq R_v\subseteq\Omega_{v+w'}$ for all $v\geq k_0$. If $v<k_0$ then $\Gamma_v\subseteq\Gamma_{k_0}\subseteq\Omega_w\subseteq\Omega_{v+w'}$. Thus $\Gamma_v\subseteq\Omega_{v+w'}$ for all $v\geq 0$.  

\noindent
Swapping $\Gamma$ and $\Omega$ and repeating the above proof gives a number $w''$ with $\Omega_v\subseteq\Gamma_{v+w''}$ for all $v\geq 0$. Set $w=\texttt{max}\{w',w''\}$. This proves $(i)$.

\noindent
$(ii)$  Let $f\in\sqrt{\texttt{Ann}_{R[\bar{d}_1,\cdots,\bar{d}_n]}(\texttt{gr}_{\Gamma}(M))}$ be homogeneous of degree $s$. There exists an integer $m\geq 1$ with $f^m\in{\texttt{Ann}_{R[\bar{d}_1,\cdots,\bar{d}_n]}(\texttt{gr}_{\Gamma}(M))}$ and an element $\beta\in\Sigma_s$ with $f=\sigma_s(\beta)$. Thus $\beta^m\Gamma_i\subseteq\Gamma_{ms+i-1}$ for every $i\geq0$. By induction on $q$ this implies that 
\begin{equation}\label{EE}
\beta^{mq}\Gamma_i\subseteq\Gamma_{i+msq-q}    
\end{equation}
for every $q\geq 1$. By $(i)$ there exists an integer $w$ such that $\Gamma_{i-w}\subseteq\Omega_i\subseteq\Gamma_{i+w}$ for all $i\geq 0$. Together with (\ref{EE}) for $q=2w+1$ this implies 
$$\beta^{m(2w+1)}\Omega_i\subseteq\beta^{m(2w+1)}\Gamma_{i+w}\subseteq\Gamma_{i+ms(2w+1)-w-1}\subseteq\Omega_{i+ms(2w+1)-1}.$$
Thus $\beta^{m(2w+1)}\Omega_i\subseteq\Omega_{i+ms(2w+1)-1}$ for all $i\geq 0$. Therefore $f^{m(2w+1)}\in\texttt{Ann}_{R[\bar{d}_1,\cdots,\bar{d}_n]}(\texttt{gr}_{\Omega}(M))$, i.e. $f\in\sqrt{\texttt{Ann}_{R[\bar{d}_1,\cdots,\bar{d}_n]}(\texttt{gr}_{\Omega}(M))}$. Hence
$\sqrt{\texttt{Ann}_{R[\bar{d}_1,\cdots,\bar{d}_n]}(\texttt{gr}_{\Gamma}(M))}\subseteq\sqrt{\texttt{Ann}_{R[\bar{d}_1,\cdots,\bar{d}_n]}(\texttt{gr}_{\Omega}(M))}$. The opposite inclusion follows similarly.
\end{proof}

Recall that if $A$ is a graded commutative Noetherian ring and $N$ is a finitely generated graded graded $A$-module, then ann$(N)$, the annihilator of $N$, is a homogeneous ideal of $A$ and the dimension of $N$ is defined to be the Krull dimension of the graded ring $A/ann(M)$, i.e. the maximum length of a chain of homogeneous prime ideals $\mathfrak p_1\subset \mathfrak p_2\subset\dots\subset\mathfrak p_d$ in the ring $A/ann(M)$.

The above proposition shows, in particular, that the dimension of $\texttt{gr}_{\Gamma}(M)$ as a graded $\texttt{gr}_\Sigma D$-module, is independent of the good filtration $\Gamma$. 

\begin{defn}
The dimension of a finitely generated $D-$module $M$, denoted $d(M)$, is the dimension of $\texttt{gr}_{\Gamma}(M)$ as a graded $\texttt{gr}_\Sigma D$-module, where $\Gamma$ is any good filtration on $M$.
\end{defn}

\begin{lemma}\label{submodule}
Let $M$ be a finitely generated $D$-module and let $M'\subset M$ be a submodule. Then $d(M')\leq d(M)$.
\end{lemma}

\begin{proof} A good filtration $\Gamma$ on $M$ induces the filtration $\Gamma'_i=\Gamma_i\cap M'$ on $M'$. The $R$-module injections $\Gamma'_i\subset \Gamma_i$ and $\Gamma'_{i-1}\subset \Gamma_{i-1}$ induce an $R$-module injection $\Gamma'_i/\Gamma'_{i-1}\subset\Gamma_i/\Gamma_{i-1}$. Hence $\texttt{gr}_{\Gamma'}(M')$ is a $\texttt{gr}_\Sigma D$-submodule of $\texttt{gr}_{\Gamma}(M)$ and so ann$(\texttt{gr}_{\Gamma'}(M'))\supset {\rm ann} (\texttt{gr}_{\Gamma}(M))$.
\end{proof}

\begin{lemma} \label{K-algebra} {\rm \cite[11.2]{AMcD}}
Let $A=A_0\oplus A_1\oplus A_2\oplus\dots$ be a commutative graded Noetherian ring with $A_0=k$ a field, generated by $A_1$ as an $A_0$-algebra. Let $N=N_0\oplus N_1\oplus N_2\oplus \dots$ be a finitely generated graded $A$-module. For sufficiently large $t$ the function $p(t)={\rm dim}_k(N_0\oplus\dots\oplus N_t)$ is a polynomial in $t$ of degree $d={{\rm dim}}N$.
\end{lemma}

\begin{lemma} \label{NoethNorm} {\rm (Noether normalization \cite[VII, Theorem 31]{ZS})}
Let $\mathfrak{p}$ be a prime ideal of $R$ of height $h$. Variables $x_1\dots,x_n$ can be chosen so that  $\mathfrak{p}\cap k[[x_{h+1},\cdots,x_n]]=0$ and $R/\mathfrak{p}$ is finite over $k[[x_{h+1},\cdots,x_n]]$. 
\end{lemma}

\section{Proof of Theorem \ref{main} }
Let $\mathfrak p\subset R$ be an associated prime of $M$, let $z\in M$ be an element of $M$ with $\texttt{Ann}_R(z)=\mathfrak{p}$, let $h={\rm height}\mathfrak p$ and assume that the variables $x_1,\cdots,x_n$ satisfy the conditions in Lemma \ref{NoethNorm}. Let $N=Dz\subset M$ be the $D$-submodule of $M$ generated by $z$. Lemma \ref{submodule} implies that it is enough to show that dim$N\geq n$.

Let $\texttt{gr}(N)=\Sigma_0z\oplus(\Sigma_1z/\Sigma_0z)\oplus(\Sigma_2z/\Sigma_1z)\oplus\cdots$ be the associated graded $\texttt{gr}(D)$-module. Clearly $\texttt{gr}(N)$ is a cyclic $\texttt{gr}(D)$-module, generated by $z\in \Sigma_0 z$. Hence $\texttt{gr}(N)\cong \texttt{gr}(D)/J$ where $J\subset \texttt{gr}(D)$ is the annihilator of $z\in \texttt{gr}(N)$. Since $\texttt{gr}(N)$ is a graded module, $J$ is a homogeneous ideal of $\texttt{gr}(D)$. 

Set $\tilde D$ to be the ring $\texttt{gr}(D)/J$. Since $\texttt{gr}(N)\cong \texttt{gr}(D)/J$, we need to show that the dimension of $\texttt{gr}(D)/J$ is at least $n$. Let $\tilde D_i$ be the degree $i$ piece of $\tilde D$, i.e. $\tilde D=\tilde D_0\oplus\tilde D_1\oplus\dots$. Let $\tilde D_+\subset \tilde D$ be the ideal generated by the elements of positive degrees. Clearly, $\tilde D_+$ is a homogeneous prime ideal of $\tilde D$, hence ${\rm dim}\tilde D\geq {\rm dim}\tilde D/\tilde D_++{\rm height}\tilde D_+.$ Since $\texttt{Ann}_R(z)=\mathfrak{p}$, it follows that  $\tilde D_0\cong R/\mathfrak{p}$. and so dim$\tilde D/\tilde D_+={\rm dim}R/\mathfrak{p}=n-h$. It follows that it is enough to prove that $${\rm height}\tilde D_+\geq h.$$

Let $S\subset \tilde D_0$ be the non-zero elements of $\tilde D_0$. Since $\tilde D_0\cong R/\mathfrak p$ and $\mathfrak p$ is a prime ideal, $S$ is a multiplicatively closed set. Let $S^{-1}\tilde D=S^{-1}\tilde D_0\oplus S^{-1}\tilde D_1\oplus \dots$ be the ring obtained from $\tilde D$ by inverting every element of $S$. Since height$\tilde D_+\geq {\rm height}S^{-1}\tilde D_+$ it is enough to prove that $${\rm height }S^{-1}\tilde D_+\geq h.$$ 

Clearly, $S^{-1}\tilde D_0\cong K$, the fraction field of $R/\mathfrak p$, and $S^{-1}\tilde D$ is a finitely generated graded $K$-algebra. By Lemma \ref{K-algebra}, the function $p(t)={\rm dim}_K(S^{-1}\tilde D_0\oplus S^{-1}\tilde D_1\oplus \dots \oplus S^{-1}\tilde D_t)$ is, for sufficiently big $t$, a polynomial in $t$ which we denote $\tilde p(t)$, and deg$\tilde p(t)={\rm dim}S^{-1}\tilde D$. It is enough to prove that $${\rm deg}\tilde p(t)\geq h.$$

Let $T=k[[x_{h+1},\cdots,x_n]]$ and let $\mathcal{K}$ be the field of fractions of $T$. By Lemma \ref{NoethNorm} $K$ is a finite field extension of $\mathcal K$. Let $d$ be the degree of this extension. Since dim$_\mathcal KL=d({\rm dim}_KL)$ for every $K$-vector space $L$, we conclude that $$q(t)={\rm dim}_{\mathcal K}(S^{-1}\tilde D_0\oplus S^{-1}\tilde D_1\oplus \dots \oplus S^{-1}\tilde D_t)=dp(t).$$ Hence $q(t)$, for sufficiently high $t$, is a polynomial in $t$ of the same degree as $\tilde p(t)$. We denote this polynomial $\tilde q(t)$. It is enough to show that $${\rm deg}\tilde q(t)\geq h.$$

Since $K$ is a finite field extension of $\mathcal K$, $S^{-1}L=\mathcal K\otimes_TL$ for every $(\tilde D_0=R/P)$-module $L$. In particular, $S^{-1}\tilde D_0\oplus S^{-1}\tilde D_1\oplus \dots \oplus S^{-1}\tilde D_t=\mathcal K\otimes_T\tilde D_0\oplus \mathcal K\otimes_T\tilde D_1\oplus\dots\mathcal K\otimes_T\tilde D_t$. But $\tilde D_i\cong \Sigma_iz/\Sigma_{i-1}z$, hence $$\tilde{q}(t)={\rm dim}_{\mathcal K}(\mathcal K\otimes_T\tilde D_0\oplus \mathcal K\otimes_T\tilde D_1\oplus\dots\mathcal K\otimes_T\tilde D_t)={\rm dim}_{\mathcal K}(\mathcal K\otimes_T(\tilde D_0\oplus\tilde D_1\oplus\dots\oplus\tilde D_t))={\rm dim}_{\mathcal K}(\mathcal K\otimes_T(\Sigma_tz)).$$ The last equality and the next lemma use the crucial fact that $\Sigma_tz$ and $N$ are $T$-modules, hence $\mathcal K\otimes_T(\Sigma_tz)$ and $\mathcal K\otimes_TN$ exist (in contrast, $\Sigma_tz$ and $N$ are not $(\tilde D_0=R/P)$-modules).

\begin{lemma}
The set $\{d_1^{t_1}\cdots d_h^{t_h}z\}\subset \mathcal K\otimes_TN$, as $t_1,\cdots, t_h$ range over all non-negative integers, is linearly independent over $\mathcal{K}$ (by a slight abuse of notation we identify the elements $d_1^{t_1}\cdots d_h^{t_h}z$ of $N$ with their images in $\mathcal K\otimes_TN$ under the natural localization map $N\to \mathcal K\otimes_TN$ that sends every $n\in N$ to $1\otimes n$).
\end{lemma}

\begin{proof}
Since $K$ is a finite extension of $\mathcal K$, let $f_i$, for every $i$ with $1\leq i\leq h$, be the monic minimal polynomial of $\bar x_i$ over $\mathcal K$, where $\bar x_i$ is the image of $x_i$ in $K$. Clearly, $f_i(x_i)\in \mathfrak  p$ and therefore $f_i(x_i)z=0$ while $f'_i(x_i)$, where $f'_i$ is the derivative of $f_i$,  is non-zero in $K$ and therefore $f_i'(x_i)z\ne 0$. 

We claim that if $s>t$ then 
\begin{equation}\label{E: second}
f_i(x_i)^sd_i^tz=0
\end{equation} 
If $t=0$ (hence $s\geq 1$), there is nothing to prove since $f_i(x_i)z=0$. By (\ref{E: first}) $$f_i(x_i)^sd_i=d_if_i(x_i)^s-sf'_i(x_i)f_i(x_i)^{s-1},$$ therefore for $t>0$ we have that
\begin{equation}\label{E: third}
f_i(x_i)^sd^t_iz=d_if_i(x_i)^sd_i^{t-1}z-sf'_i(x_i)f_i(x_i)^{s-1}d_i^{t-1}z=0,
\end{equation}
where both summands in the middle vanish by induction on $t$. This proves the claim.

Equalities (\ref{E: second}) and (\ref{E: third}) imply by induction on $t$ that 
\begin{equation}\label{E: fourth}
f_i(x_i)^td_i^tz=(-1)^tt!f'_i(x_i)^tz\ne 0.
\end{equation} 
Now let $$\gamma=\Sigma_{t_1,\dots, t_h} c_{t_1,\dots, t_h}d_1^{t_1}\cdots d_h^{t_h}z,$$ where $c_{t_1,\cdots, t_h}\in \mathcal K$ be a linear combination of finitely many elements of the set $\{d_1^{t_1}\cdots d_h^{t_h}z\}$. Let $\{\tau_1,\dots, \tau_h\}$ be an index of highest total degree $\tau_1+\dots+\tau_h$. Every other $c_{t_1,\dots, t_h}d_1^{t_1}\cdots d_h^{t_h}z$ in this linear combination has some $t_j$ with $t_j<\tau_j$, hence $f_j(x_j)^{\tau_j}c_{i_1\dots,i_n}d^{i_1}\cdots d_n^{i_n}=0$ and  $$f_1(x_1)^{\tau_1}\cdots f_h(x_h)^{\tau_h}\gamma=c_{\tau_1,\dots,\tau_h}f_1(x_1)^{\tau_1}\cdots f_h(x_h)^{\tau_h}d_1^{\tau_1}\cdots d_h^{\tau_h}z$$$$=(-1)^{\tau_1+\dots+\tau_n}t_1!\cdots t_n!c_{\tau_1,\dots,\tau_h}f_1'(x_1)^{\tau_1}\cdots f_h'(x_h)^{\tau_h}z\ne 0,$$ where we use (\ref{E: second}), (\ref{E: fourth}) and the fact that $f_i(x_i)^{\tau_i}$ and $f'_i(x_i)^{\tau_j}$ commute with every $d_j^{\tau_j}$ with $j\ne i$. Thereofre $\gamma\ne 0$. 
\end{proof}

The number of the elements $\{d_1^{t_1}\cdots d_h^{t_h}z\}$, as $d_1+\dots+d_h\leq t$, is the number of monomials in $h$ variables of total degree at most $t$, which equals ${t+h\choose h}$. Since these elements are in $\Sigma_tz$ and are linearly independent, $\tilde q(t)\geq {t+h\choose h}$ for sufficiently high $t$. But $\tilde q(t)$ is a polynomial in $t$ and ${t+h\choose h}$ is a polynomial in $t$ of degree $h$.  Hence the degree of $\tilde q(t)$ is at least $h$. This completes the proof of Theorem \ref{main}. \qed

\newpage

\end{document}